\documentclass[12pt]{amsart}

\usepackage{graphics}  %%% Lucio
\usepackage{pdfpages} %%%  Luigi
\usepackage{ulem} %{para riscar palavras ao meio}
\usepackage{amsmath,amsthm,amsopn,amsfonts,amssymb,color}
\usepackage[colorlinks=true,linkcolor=blue,urlcolor=blue]{hyperref} 
\usepackage{cancel}
 
\usepackage{graphicx}
\usepackage{float}
\usepackage{amsmath}
\usepackage{fancyhdr}
\usepackage{epstopdf}
\usepackage{amsfonts}

\usepackage{booktabs}
\usepackage[USenglish]{babel}
\usepackage{cancel}
\usepackage{times}
\usepackage{mathrsfs}
\usepackage{multirow,multicol}
\usepackage{amsfonts}
\usepackage{amsthm}
\usepackage{amsmath}
\usepackage{amsmath,amssymb}
\newtheorem{teo}{Theorem}[section]

\newtheorem{defin}[teo]{Definition}
\newtheorem{remark}[teo]{Remark}

\newtheorem{lemma}[teo]{Lemma}

\newcommand\R{\mathbb{R}}

\def\elle#1{L^{#1}(\Omega)}

\def\div{{\rm div}}
\def\elle#1{L^{#1}(\Omega)}
\def\w{H_0^{1}(\Omega)}
\def\io{\int_{\Omega}}
\def\norma#1#2{\|#1\|_{\lower 4pt \hbox{$\scriptstyle #2$}}}

\def\D{\nabla}

\def\finedim
\def\gw{G_{\tilde{k}}(w_n)}

\def\R{I \!\!R}

\def\elle#1{L^{#1}(\Omega)}

\def\w{W_0^{1,2}(\Omega)}

\def\w1{W_0^{1,1}(\Omega)}

\def\eps{\varepsilon}

\def\dys{\displaystyle}

\def\w{H_0^{1}(\Omega)}

\def\be{\begin{equation}}
\def\ee{\end{equation}}

\def\vn{u_n}

 \usepackage{color}
 \numberwithin{equation}{section}

\usepackage{todonotes}

\title[]{ A singular Schr\"odinger-Maxwell system}
\author{Lucio Boccardo}
\address{Istituto Lombardo \& Sapienza Università di Roma, Italy.}
\email{boccardo@mat.uniroma1.it}

\author{Stefano Buccheri}

\address{Faculty of Mathematics - University of Vienna
\hfill\break \indent Oskar-Morgenstern-Platz 1, 1090 Vienna, Austria.}
\email{stefano.buccheri@univie.ac.at}

\author{Carlos Alberto dos Santos}

\address{Departamento de Matem\'atica - Instituto de Ci\^{e}ncias Exatas - Universidade de Bras\'{i}lia
\hfill\break \indent Campus Universit\'{a}rio Darcy Ribeiro, 70910-900, Bras\'{i}lia - DF - Brazil.}
\email{c.a.p.santos@mat.unb.br}


\date{\today}

\begin{document}

\maketitle 

\begin{abstract}
In this paper we are concerned with existence of positive solutions for a Schr\"odinger-Maxwell system with singular or strongly-singular terms.  We overcome the difficulty given by the singular terms through an approximation scheme and controlling the approximated sequences of solutions with suitable barriers from above and from below.  Besides this, in some particular case, we show that the unique energy solution of the singular system is a saddle point of a suitable functional.
\end{abstract}

\tableofcontents
\section{Introduction}
In this paper we will focus on the following system with singular nonlinearity

\begin{equation} \label{sing}
\dys
\begin{cases}
-\div(A(x)u)+v u^{r-1}=\dys \frac{1}{u^{\gamma}} \qquad & \mbox{in } \Omega\,,\\
-\div(M(x)v) = u^r \qquad & \mbox{in } \Omega\,,\\
u,v>0 \qquad & \mbox{in } \Omega\,,\\
u = \varphi =0 & \mbox{on }  \partial \Omega\,,
\end{cases}
\end{equation}
where $\Omega$ is an open bounded set of $\mathbb{R}^N$, with $N \geq 2$, $r,\gamma>0$ are suitable constants, and the measurable matrices $A(x), M(x)$ are bounded and elliptic in the sense that 
\begin{equation}\label{alpha}
\alpha|\xi|^2\le A(x)\, \xi\,\xi \le \beta|\xi|^2 \ \ \ \mbox{and} \ \ \ \alpha|\xi|^2\le M(x)\, \xi\,\xi  \le\beta|\xi|^2 \,,
\end{equation}
for almost every $x$ in $\Omega$, and for every $\xi$ in $\R^{N}$, with $0 < \alpha \leq \beta$.\\

To place our work within the existing literature, we will briefly recall some known facts on Schr\"odinger-Maxwell systems and singular equations. Being impossible to give a detailed survey on these two wide themes, we just recall those results that are related to our work.
Let us start with the seminal work \cite{bf}, where Benci and Fortunato considered the following system
\begin{equation} \label{sing1}
\dys
\begin{cases}
-\frac{1}{2} \Delta u+v u=\dys \omega u \qquad & \mbox{in } \mathbb{R}^N\,,\\
-\Delta v= 4 \pi u^2 \qquad & \mbox{in }  \mathbb{R}^N,
\end{cases}
\end{equation}
that comes from the study of the eigenvalue problem for the Schr\"odinger operator when coupled with an electromagnetic field.
Using the fact that solutions of \eqref{sing1} can be characterized as critical point of a suitable functional (not bounded neither from below nor from above), the authors prove that \eqref{sing1} admits an increasing and divergent sequence of eigenvalues $\{\omega_n\}$.\\
More recently, taking inspiration from the structure of \eqref{sing1}, a series of papers (see for instance \cite{bc,bo,du}) studied the existence and the regularizing effect of the problem
\begin{equation} \label{sing0}
\dys
\begin{cases}
-\div(A(x)u)+Bv |u|^{r-2}u=f(x)\qquad & \mbox{in } \Omega\,,\\
-\div(M(x)v) = |u|^r \qquad & \mbox{in } \Omega\,,\\
u = \varphi =0 & \mbox{on }  \partial \Omega\,,
\end{cases}
\end{equation}
where $B>0$, $f \in L^m(\Omega)$ with $m>1$, and $A(x),M(x)$ satisfies \eqref{alpha}. One of the main feature of \eqref{sing0} is that the interplay of the two equation enhances the regularizing effect of the system with respect of the one of the single equation. The main techniques used in \cite{bc,bo,du} are approximation scheme, a priori estimates through a test-function-based approach and fixed point theorems. These tools can be used for more general system that do not necessarily have a variational structure.\\

On the other hand, the literature about singular equation is wide and well established. Without the intention of being exhaustive, we mention the seminal works \cite{bo,1, 2}. We stress that in the previously mentioned paper the singularity is of reaction type, meaning something like
\[
-\mbox{div}(A(x)u)=\displaystyle h(u) \ \ \ \mbox{with} \ \ \ h(u)\approx\frac{1}{|u|^{\gamma}} \ \ \ \mbox{near the origin}. 
\]
It is well known that if $\gamma$ is small enough ($\gamma\in(0,1]$ or $\gamma\in(0,3)$ plus some additional smoothness assumption on the data) such type of equation admits a solution in $W^{1,2}_0(\Omega)$. On the other hand, if $\gamma\ge 3$, solutions still exists but they belong just to $W^{1,2}_{loc}(\Omega)$. For more information on this topic we refer to the introduction of \cite{OP} and reference therein.\\
We briefly mention that if the singularity is of absorption type, namely something like
\[
-\mbox{div}(A(x)u)+\displaystyle h(u)=1 \ \ \ \mbox{with} \ \ \ h(u)\approx\frac{1}{|u|^{\gamma}} \ \ \ \mbox{near the origin},
\]
the features of the equation change dramatically. We address the interested reader to \cite{choi,dav,diazbis}.\\

Problems of the type of \eqref{sing0}, with the source term $f(x)$ replaced with the nonlinearity  $u^{-\gamma}$, have not been considered on the literature that much. Some related results can be found in  \cite{cgh,yc,zq,zq1}, in general for weak singularities ($\gamma\in(0,1)$ with the exception of \cite{zq1}), and based on a variational framework. Other results for other type of system with singular nonlinearities can be found for instance in  \cite{arruda,cav,zhang}.\\

Motivated by the above context, the main aim of this work is to study the singular Schr\"odinger-Maxwell system \eqref{sing}, including the non variational case and the strongly singular one. More in detail we provide different regimes for the parameters $r$ and $\gamma$ in order to have a existence of a solution for \eqref{sing}.\\

When $\gamma\in (0,1)$, or $\gamma\in (0,3)$ and some additional smoothness assumptions are made on $\partial\Omega$ and $A(x)$, and $r>\max\{0,1-\gamma\}$ we provide existence of an {\it energy solution} for our system (see Definition \ref{energysol} below). We stress that when $\gamma\in(0,1)$ even the nonlinear term in the left hand side of the first equation in \eqref{sing} can be singular and so we have a double singularity system (this is something that has not been considered in the existing literature). Anyway, this singularity is mild compared with the one on the right hand side. In broad terms, due to the assumption on  $r$, our singularity is always of reaction type.\\

If $\gamma$ is large, we do not expect existence of energy solutions, and then we introduce a suitable notion of {\it distributional solution} for \eqref{sing} (see Definition \ref{15}) and prove existence. We also need an alternative way to proscribe the zero boundary condition to the solution of the singular equation.\\

In some particular case, even if our approach to prove existence is not variational, we are able to show that the obtained energy solution is a saddle point of a suitable functional unbounded from above and from below. This somehow reflect the structure of the original problem \eqref{sing1}. Moreover we provide a partial uniqueness result.\\

Let us focus at first on the case $\gamma\in(0,1]$.
\begin{defin}\label{energysol}
A couple of functions $(u,v)\in \left(W^{1,2}_0(\Omega)\cap \elle{\infty}\right)^2$ is a energy solution to system \eqref{sing} if 
\[
u,v> 0 \ \ \  \mbox{ in } \ \Omega,
\]
\[
\frac{\phi}{u^{\gamma}}\in \elle1 \ \ \ \forall \ \phi\in W^{1,2}_0(\Omega),
\]
and hold
\begin{equation}\label{energysoleq}
\begin{cases} 
\medskip
\medskip

\displaystyle \io A(x)\nabla u\nabla\phi+\io v u^{r-1} \phi=\io\frac{\phi}{u^{\gamma}} & \ \ \ \forall \ \phi\in W^{1,2}_0(\Omega) \\

\medskip

\displaystyle  \io  M(x)\nabla v \nabla \psi=\io u^{r}\psi & \ \ \ \forall \ \psi\in W^{1,2}_0(\Omega) .
\end{cases}
\end{equation}
\end{defin}

\begin{teo} \label{existence1}
Let us assume \eqref{alpha} and take $\gamma\in(0,1]$ and $r>1-\gamma$. there exists $(u,v)\in \left(W^{1,2}_0(\Omega)\cap L^{\infty}(\Omega)\right)^2$ energy solution to \eqref{sing}. Moreover:
\begin{enumerate}
\item[$i$)] if $r\geq 1$ such a solution is unique,
\item[$ii$)] if $0<\gamma <1$ and $r\geq 1$, the unique solution $(u,v)$  is a saddle point to the functional.
\end{enumerate}
\[
J(w,z)=\frac12\int_{\Omega} A(x)\nabla w\nabla w-\frac1{2r}\int_{\Omega}M (x)\nabla z\nabla z
+\frac1{r}\int_{\Omega}  z^+ |w|^r
-\frac1{1-\gamma} \int_{\Omega} \bigl(w^{+} \bigr)^{1-\gamma}
\]
that is, 
\begin{equation}\label{16}
J(u,z)\le J(u,v) \le J(w,v) \ \mbox{ for any } \ (w,z)\in \left(W^{1,2}_0(\Omega)\right)^2.
\end{equation}
\end{teo}
\noindent The first step to prove Theorem \ref{existence1} is to focus on the single equation
\[
-\div(A(x)u)+v u^{r-1}=\dys \frac{1}{u^{\gamma}}
\]
with a {\it frozen} potential $v$. One important feature of such an equation is that it enjoys a comparison principle for solutions with small $L^{\infty}-$norm.
Then, through an approximation procedure, we prove existence and uniqueness of a solution $u$ that satisfies some a priory estimates that do not depend on $v$. This latter property will allow us to apply the Schauder Fixed point Theorem to the whole system and obtain existence.
If $r\ge1$, we are able to generalize the argument of \cite{cgh} to prove that the solution is unique. Unfortunately we are not able to provide uniqueness in  the case $r\in(0,1)$. As far as the functional is concerned, note that if $(u,v)$ is an energy solution of the system then $J(u,v)$ is finite and the chain of inequality \eqref{16} is well defined. It is easy to verify (just take $w\in W^{1,2}_0(\Omega)\cap \elle{\infty}$ and evaluate  $J(tw,v)$ and $J(u,tw)$ as the real parameter $t$ diverges) that the functional is unbounded from both above and below.\\

The same approach (for existence and uniqueness) can be applied in the case $\gamma\in(1,3)$, paying the price, however, of more regularity on both $\partial \Omega$ and the matrix $A(x)$. More in detail we shall additionally assume that
\be\label{30.08}
\partial\Omega\in C^{2,1} \ \ \ \mbox{and} \ \ \ A_{i,j}\in C^{0,1}(\overline{\Omega}) \ \ \mbox{for} \ \ i,j=1,\cdots, N
\ee
We have the following result.
\begin{teo} \label{30.08x}
Let us assume \eqref{alpha} and \eqref{30.08}. Given $\gamma\in(1,3)$ and $r>0$, there exists $(u,v)\in \left(W^{1,2}_0(\Omega)\cap L(\Omega)^{\infty}\right)^2$ energy solution to \eqref{sing} such that
\be\label{sun}
u(x)\ge c d(x)^{{2}/{(\gamma+1)}} \ \ \ \mbox{in} \ \ \ \Omega
\ee
for some $c>0$, where $d(x)$ stands for the distance of $x \in \Omega$ to the boundary of $\Omega$. Moreover if $r\ge1$ the solution is unique.
\end{teo}
\noindent The additional assumption \eqref{30.08} allows to prove the pointwise estimate from below \eqref{sun}, using a suitable subsolution argument (here again the comparison principle plays its role). Such estimate enables us to take advantage of the Hardy Inequality
\[
\mathcal{H}\io \left(\frac{\phi}{d(x)}\right)^2\le \io|\nabla \phi|^2 \ \ \ \forall \ \ \phi\in W^{1,2}_0(\Omega),
\]
and prove the required energy a priory estimate for the sequence of approximating solutions.\\

On the other hand, if $\gamma\ge 3$ we do not expect to have finite energy solutions, independently on the smoothness of the data, and a more general notion of distributional solution is needed.

\begin{defin}\label{15} We say that the system \eqref{sing} admits a distributional solution if there exsit $u\in W_{loc}^{1,2}(\Omega)\cap\elle{\infty}$ and $v\in W_0^{1,2}(\Omega)\cap\elle{\infty}$ such that
\[
u,v> 0 \ \ \ a.e. \mbox{ in } \ \Omega,
\]
\[
(u-\epsilon)^+ \in W^{1,2}_0(\Omega) \ \ \mbox{ for any } \ \epsilon>0,
\]
\[
\frac{\phi}{u^{\gamma}}\in \elle1 \ \ \ \forall \ \phi\in C^1_c(\Omega),
\]
and 
\[
\begin{cases} 
\medskip
\medskip
\displaystyle \io A(x)\nabla u\nabla\phi+\io v u^{r-1} \phi=\io\frac{\phi}{u^{\gamma}} \\
\medskip
\displaystyle  \io M(x)\nabla v \nabla \psi=\io u^r\psi
\end{cases}
\ \ \ \forall \ (\phi, \psi)\in C^{1}_{c}(\Omega) \times W^{1,2}_0(\Omega).
\]
\end{defin}
So, we have the following result.
\begin{teo}\label{teo15}
Let us assume \eqref{alpha} and \eqref{30.08}. 
Assume that $\gamma\geq 1$ and $r>0$. Then there exists a couple of positive functions $(u,v)\in \big(W^{1,2}_{loc}(\Omega)\cap L^{\infty}(\Omega)\big) \times \big(W^{1,2}_{0}(\Omega)\cap L^{\infty}(\Omega)\big)$ that solves \eqref{sing} in the sense of definition \ref{15}. 
\end{teo}

\noindent To prove Theorem \ref{teo15} do not start solving the single singular equation. Instead, always through a fixed point theorem, we built a sequence of approximating solution $(u_n,v_n)$ for a suitable regularized system. Then we obtain apriori estimates (local for $u_n$ and global for $v_n$) and finally pass to the limit.\\

This paper is organized as follows. In Section 2, we show existence of solutions,  positiveness   and a comparison result for a single singular modified problem, while in Section 3, we prove existence of energy solutions to problem \eqref{sing} with $0<\gamma <3$. In section 4, we show that the unique solution of \eqref{sing} is a saddle point when  $0<\gamma <1$, while in Section 5, we show existence of solutions outside the energy space.

\section{Preliminary results}

In this section, we present some auxiliary results related to the problem \eqref{sing}. We begin by showing existence of a sequence of solutions, uniformly bounded above by an appropriate power of the $L^{\infty}-$norm of the potential term $v$.

\begin{lemma}\label{16:09}Let $\Omega$ be a bounded open set of $\mathbb{R}^N$ and $A$ a matrix that satisfies \eqref{alpha}. Let us assume $\gamma>0$, $r>1-\gamma $ and $0\le v \in\elle{\infty}$. Then, for any $n\in \mathbb N$, there exists $0\le u_n\in  W^{1,2}_0(\Omega)\cap\elle{\infty}$ such that
\begin{equation}\label{15:56}
\io A(x)\nabla u_n \nabla \phi=\io \frac{1}{(\frac1n+u_n)^{\gamma}}(1- v  u_n^{\gamma+r-1})\phi, \ \ \ \ \  \forall\, \phi\in W^{1,2}_0(\Omega).
\end{equation}
Moreover there exists $C_0=C_0(\alpha,N,|\Omega|)$ such that
\begin{equation}\label{14;25}
\|u_n\|_{\elle{\infty}}\le \min\{\| v \|_{\elle{\infty}}^{\frac1{1-r-\gamma}},C_0\}.
\end{equation}
\end{lemma}
\begin{proof}
Let us define a map $F:W^{1,2}_0(\Omega)\to W^{1,2}_0(\Omega)$ such that $F(w)=u_w$ is the unique solution to
\be\label{auxx}
-\div(A(x) u_w)=\frac{1}{(\frac1n+|w|)^{\gamma}}(1- v  |w|^{\gamma+r-1})\chi_{\{|w|\le b\}},
\ee
with $b=\| v \|_{\elle{\infty}}^{\frac1{1-r-\gamma}}$. We stress that the left hand side above satisfies
\[
0\le \frac{1}{(\frac1n+|w|)^{\gamma}}(1- v  |w|^{\gamma+r-1})\chi_{\{|w|\le b\}}\le n.
\] 
This easily imply that the map $w\mapsto F(w)$ is continuous. Indeed, if $\{w_m\}\subset W^{1,2}_0(\Omega)$ is a sequence that strongly converges to $w$ in $W^{1,2}_0(\Omega)$, then the sequence of the images $u_{w_m}=F(w_m)$ strongly converges in $W^{1,2}_0(\Omega)$ to $u_w=F(w)$ the unique solution of \eqref{auxx} with datum $w$ (recall that $n$ is fixed at this stage). Similarly it follows that the map is completely continuous: if $w_m$ is bounded in $W^{1,2}_0(\Omega)$, up to a subsequence, $w_m$ strongly converges in $\elle 2$ to some $w\in W^{1,2}_0(\Omega)$, and then we infer that  $v_{w_m}=T(w_m)$ strongly converges in $W^{1,2}_0(\Omega)$  to  $T(w)=u_w$.  Moreover, by taking  $u_w$ as a test function in \eqref{auxx}, we get
\[
\left(\io|\nabla u_w|^2\right)^{\frac12}\le \frac n{\mathcal S\alpha}|\Omega|^{1-\frac1{2^*}}.
\]
Then, for any fixed $n\in\mathbb{N}$, there exist a fixed point $u_n=T(u_n)\in W^{1,2}_0(\Omega)$ that solves 
\be\label{10:44}
-\div(A(x) u_n)=\frac{1}{(\frac1n+|u_n|)^{\gamma}}(1- v  |u_n|^{\gamma+r-1})\chi_{\{|u_n|\le b\}},
\ee
in the energy sense. Since the right hand side is positive (thanks to the choice of $b$), it follows that $u_n$ is positive. To prove the $L^{\infty}$-bound \eqref{14;25}, let us recall the definition of
\begin{equation}
\label{25a}
T_k(s)=\begin{cases}
\displaystyle
-k,~~s\leq -k,\\
s,~~\vert s \vert \leq k,\\
k,~~ s \geq k,
\end{cases}
\end{equation}
and $G_k(s) = s-T_k(s)$ for $s\in \mathbb{R}$ and $k>0$. So, taking $G_b(u_n)=\max\{b,u_n\}-b$  as a test function in the weak formulation of the equation above and noting that $G_b(u_n)\equiv 0$ on the set $\{0\le u_n\le b\}$, we obtain 
\[
\alpha\int_{u_n>b}|\nabla u_n|^2=\alpha\io|\nabla G_b(u_n)|^2\le \int_{\{0\le u_n\le b\}}\frac{1}{(\frac1n+u_n)^{\gamma}}(1- v  u_n^{\gamma+r-1}) G_b(u_n)=0.
\]
Then we deduce that 
\[
|u_n|\le b=\| v \|_{\elle{\infty}}^{\frac1{1-r-\gamma}}.
\]
Moreover, taking $G_k(u_n)$, $k\ge1$, as a test function in \eqref{10:44} again, and disregarding the negative contribution depending on $ v $, we also get
\[
\alpha\io|\nabla G_k(u_n)|^2\le \io G_k(u_n).
\]
Taking advantage of Th\'eor\`eme 4.2 in \cite{S66}, we then deduce that there exists a constant $C_0=C_0(\alpha,N,\Omega)$ such that $\|u_n\|_{\elle{\infty}}\le C_0$ (for a more recent reference see Theorem 6.6 of \cite{bc}). Then we have $\|u_n\|_{\elle{\infty}}\le \min\{\| v \|_{\elle{\infty}}^{\frac1{1-r-\gamma}},C_0\}$. This ends the proof.\\
\end{proof}

The below Lemma will permit us prove the positiveness of the solution of the problem \eqref{sing} for some variations of $\gamma>0$.

\begin{lemma}\label{16:51}
Let $\Omega$ be a bounded open set of $\mathbb{R}^N$  and let $B>0$  a real number. Assume $\gamma>0$, $r>1-\gamma $, and $u\in W_0^{1,2}(\Omega)$, with $0\le u\le B^{\frac{1}{1-r-\gamma}}$, satisfies
\[
\io A(x)\nabla u \nabla \phi\ge \io \frac{1}{(1+u)^{\gamma}}(1-B u^{\gamma+r-1})\phi \ \ \  \forall \phi\in W^{1,2}_0(\Omega).
\]
Then, for any $\omega\subset\subset \Omega$, there exists a constant $c_{\omega,B}>0$ such that
\be\label{bfb}
 \ u\ge c_{\omega,B} \ \ \ \mbox{in} \ \ \ \omega.
\ee
\end{lemma}

\begin{proof}
Let us consider the unique $w\in W_0^{1,2}(\Omega)$ that satisfies
\be\label{11;07}
\io A(x)\nabla w\nabla \phi= \io \frac{1}{(1+w)^{\gamma}}(1-B w^{\gamma+r-1})\phi \ \ \ \ \  \forall \phi\in W^{1,2}_0(\Omega)
\ee
and such that $0\le w\le  B^{\frac{1}{1-r-\gamma}}$. Existence and uniqueness of such a function is provided by Lemma \ref{16:09} with the special  choices $v=B$ and $n=1$. Choosing $(w-u)_+=\max\{(w-u),0\}$ as a test function in the weak formulations of $w$ and $u$ and taking the difference, we get
\be\label{11:57}
\alpha\io |\nabla (w-u)_+|^2\leq \io \left[\frac{(1-B w^{\gamma+r-1})}{(c+w)^{\gamma}}-\frac{(1-Bu^{\gamma+r-1})}{(c+u)^{\gamma}}\right](w-u)_+
\ee
\[
\le\io \left[\frac{(1-B w^{\gamma+r-1})}{(1+w)^{\gamma}}-\frac{(1-B u^{\gamma+r-1})}{(1+u)^{\gamma}}\right](w-u)_+.
\]
Since the function
\[
s\to \frac{1-B s^{\gamma+r-1}}{(1+s)^{\gamma}}
\]
is strictly decreasing on $0<s\le B^{\frac1{1-r-\gamma}}$ and both $w,v$ are smaller then $B^{\frac1{1-r-\gamma}}$, we have that the right hand side of \eqref{11:57} is negative. This implies that $(w-u)_+\equiv 0$ so that 
\[
w\le  u.
\]
Since $w\le  B^{\frac{1}{1-r-\gamma}}$, we deduce from \eqref{11;07} that
\[
\io A(x)\nabla w\nabla \phi\ge0 \ \ \ \ \  \forall \phi\in W^{1,2}_0(\Omega), \ \ \ \phi\ge0.
\]
We can then apply the Strong Maximum Principle to $w$ to infer that for any $\omega\subset\subset \Omega$ there exists a constant $c_{\omega,B}>0$ such that
\[
c_{\omega,B}\le w.
\]
Combining together the two previous inequality, we conclude the proof of the Lemma.\\
\end{proof}

%\begin{lemma}
%For all $u\in W^{1,2}_0(\Omega)$ there exists $\delta,C>0$ and $\tau>\frac12$ such that
%\be\label{hardy}
%u(x)\le Cd(x)^{\tau} \ \ \ \mbox{in} \ \ \ \Omega_{\delta}=\{x\in\Omega \ : \ d(x)<\delta\}.
%\ee
%\end{lemma}
%\begin{proof}
%Let us assume by contradiction that there exsits $\bar u\in W^{1,2}_0(\Omega)$ such that for any $\delta,C>0$ and $\tau>\frac12$ we have the opposite inequality then \eqref{hardy}. Then we can chose $\delta_n=\bar \delta=\frac12\inf_{x\in\Omega} d$, $C_n=1$ and $\tau_n>\frac12$ with $\tau_n\to\frac12$ such that $u(x)> d(x)^{\tau_n} \ \ \ \mbox{in} \ \ \ \Omega_{\bar \delta}$. Taking the limit as $n\to \infty$ we get
%\[
%u(x)\ge d(x)^{\frac12} \ \ \ \mbox{in} \ \ \ \Omega_{\bar \delta}.
%\]
%At this point thanks to the regularity assumption on $\partial\Omega$ we are entitled to combine Hardy inequality with the inequality above and get
%\[
%\int _{\Omega_{\bar \delta}} d(x)^{-1}dx\le\io \left(\frac{\bar u(x)}{d(x)}\right)^2dx \le \mathcal H \io |\nabla \bar u(x)|^2dx
%\]
%that is a contradiction since $d^{-1}$ is not an $\elle 1$ function.
%\end{proof}
The following Lemma provides us with a refinement of the bound from below \eqref{bfb}, in the case of smoother assumptions on $\partial\Omega$ and on the entries of the matrix $A(x)$.\\
Here and in the following  we will set $d(x)=$dist$(x,\partial\Omega)$, the distance of $x\in\Omega$ to the boundary.
\begin{lemma}\label{regular} Let us assume \eqref{alpha}, \eqref{30.08}, take $\gamma >1$, $r>1-\gamma$, set $\tau={2}/{(\gamma+1)}$ and let $u_n$ be the solution to \eqref{15:56}. Then we have that
\be\label{tenerifebis}
c\,d(x)^{\tau}-\frac1n \le \vn (x) ,~x \in \Omega,
\ee
where   $c=c(\|v \|_{\elle{\infty}})$ decreases as $\|v\|_{\elle{\infty}}$ increases.
\end{lemma}

\begin{proof}
Thanks to the regularity assumption \eqref{30.08}, there exist $\lambda_1>0$ and $\varphi_1\in W^{1,2}_0(\Omega)$ with $\|\varphi_1\|_{\elle{\infty}}=1$ and $|\nabla \varphi_1|\in \elle{\infty}$ such that
\[
\dys
\begin{cases}
-\div(A(x)\varphi_1)=\dys \lambda_1\varphi_1 \qquad & \mbox{in } \Omega\,,\\
\varphi_1 =0 & \mbox{on }  \partial \Omega\,;
\end{cases}
\]
moreover there exist two constant $c_1<c_2$ such that
\be\label{30.08bis}
c_1 d(x)\le \varphi_1(x)\le c_2 d(x);
\ee
see for instance Lemma 2 of \cite{diaz}. Let us set now
\[
 w_n =c_3\varphi_1(x)^{\tau}-\frac1n, \ \ \ \ \mbox{with} \ \  \ \tau=\frac{2}{\gamma+1},
\]
where $c_3>0$ is such that
\be\label{recall}
\lambda_1c_3^{1+\gamma}\tau+\beta c_3^{1+\gamma}\tau(1-\tau)\|\nabla \varphi_1\|_{\elle{\infty}}^2+\| v \|_{\elle{\infty}} c_3^{r-1+\gamma}-1\le0.
\ee
Let us stress that $c_3$ has an inverse dependence with respect to $\| v \|_{\elle{\infty}}$, namely if $\| v \|_{\elle{\infty}}$ increases $c_3$ has to be smaller, and that, in the case under consideration, $1-\tau>0$.\\
We claim that $w_n$ is a subsolution to \eqref{15:56}, namely
\be\label{8.7}
I:=\io A(x) \nabla w_n \nabla \phi-\io\frac{1}{(\frac1n+ w_n)^{\gamma}}\big(1- v   w_n^{\gamma+r-1} \big)\phi\le 0 \ \ \ \forall \ \phi\in W^{1,2}_0(\Omega), \ \phi\ge0.
\ee
Once the claim is proved, we can take $(w_n-u_n)_+$ as test a function in \eqref{8.7} and in \eqref{15:56} and proceed to their difference to obtain
\[
\alpha\io |\nabla (w_n-u_n)_+|^2\le \io \left[\frac{(1- v   w_n^{\gamma+r-1} )}{(\frac1n+ w_n)^{\gamma}}-\frac{(1-vu_n^{\gamma+r-1})}{(\frac1n+u_n)^{\gamma}}\right](w_n-u_n)_+
\] 
\[
=\int_{w_n\ge u_n} \left[\frac{(1- v   w_n^{\gamma+r-1} )}{(\frac1n+ w_n)^{\gamma}}-\frac{(1-vu_n^{\gamma+r-1})}{(\frac1n+u_n)^{\gamma}}\right](w_n-u_n)_+\le 0,
\]
where the last inequality comes from the facts that, for any $n>0$ and almost everywhere $x \in \Omega$, the function
\[
s\to \frac{1- v (x)s^{\gamma+r-1}}{(\frac 1n+s)^{\gamma}} \ \ \    s\in (0, \| v \|_{\elle{\infty}}^{\frac1{1-r-\gamma}} )
\]
is strictly decreasing, and both $w_n, u_n$ are smaller then $\| v \|_{\elle{\infty}}^{\frac1{1-r-\gamma}}$ 
(thanks to the choice of $c_3$ and to Lemma \ref{16:09} respectively). This implies $(w_n-u_n)_+\equiv0$, namely $w_n\le u_n$ almost everywhere in $\Omega$. Using \eqref{30.08bis} we conclude the proof of the Lemma.\\

To prove the claim let us take, at first, $0\le \phi\in C^1_c(\Omega)$ and notice that $w_n^{\tau-1}\phi$ is an admissible test function for \eqref{8.7}. 
Standard computation shows that (recall that $1-\tau>0$)
\[
\io A(x) \nabla w_n \nabla \phi=c_3\tau\io A(x)\varphi_1^{\tau-1} \nabla \varphi_1 \nabla \phi
\]
\[
= c_3\tau\io \big(A(x) \nabla \varphi_1\nabla (\varphi_1^{\tau-1}\phi)+(1-\tau)A(x)\nabla\varphi_1\nabla\varphi_1 \varphi_1^{\tau-2}\phi\big)
\]
\[
\le c_3\tau\io \big(\lambda_1 \varphi_1^{\tau}+\beta(1-\tau)|\nabla\varphi_1|^2 \varphi_1^{\tau-2}\big)\phi.
\]
On the other hand,
\[
-\frac{1}{(\frac1n+ w_n)^{\gamma}}\big(1- v   w_n^{\gamma+r-1} \big)\le -\frac{1}{(c_3\varphi_1^{\tau})^{\gamma}}\big(1- v   (c_3\varphi_1^{\tau})^{\gamma+r-1} \big).
\]
Plugging together the previous information we obtain
\[
I
\le\io \big(\lambda_1c_3^{1+\gamma}\tau+\beta c_3^{1+\gamma}\tau(1-\tau)\|\nabla \varphi_1\|_{\elle{\infty}}^2+\| v \|_{\elle{\infty}} c_3^{\gamma+r-1}-1\big)c_3^{-\gamma}\varphi_1^{\tau-2}\phi\le0,
\]
where the last inequality comes from the fact that $\|\varphi_1\|_{\elle{\infty}}\le1$ and the choice of the constant $c_3$. This proves that \eqref{8.7} holds true for any $0\le \phi\in C^1_c(\Omega)$. To conclude the proof of the claim, it is sufficient to use density arguments. This ends the proof.

\end{proof}

\section{Finite energy solutions}

In this section, we prove existence of energy solutions for the problem \eqref{sing}. We begin proving existence of solution for the single singular equation of the system and a given nonegative function $v\in\elle{\infty}$.

\begin{lemma}\label{primo0}
Let $\Omega$ be a bounded open set of $\mathbb{R}^N$. Assume $\gamma\in(0,1]$, $r>1-\gamma$ and $0\le v\in\elle{\infty}$.
Then there exists a unique $u\in W^{1,2}_0(\Omega)\cap\elle{\infty}$ such that
\begin{equation}\label{bound}
\frac{\phi}{u^{\gamma}}\in\elle 1 \ \ \ \  \|u\|_{\elle{\infty}}\le \min\{\|v\|_{\elle{\infty}}^{\frac1{1-r-\gamma}},C_0\},\ \ \ \ \|u\|_{W^{1,2}_0(\Omega)}\le  |\Omega|C_0^{1-\gamma},
\end{equation}
\be\label{22:03}
\ \ \mbox{and} \ \ \ \ \ \ \io A(x)\nabla u\nabla\phi+\io v u^{r-1}\phi =\io\frac{\phi}{u^{\gamma}} \ \ \ \forall \ \phi\in W^{1,2}_0(\Omega),
\ee
for any $\phi\in W^{1,2}_0(\Omega)$, where $C_0$ is the constant given in Lemma \ref{16:09}.
\end{lemma}
\begin{remark}
Let us stress that the uniqueness refers to the class of solutions to \eqref{22:03} that further satisfy the $\elle{\infty}$ bound in \eqref{bound}.
\end{remark}
\begin{proof} 
Let $\{u_n\}$ be the sequence of solutions given by Lemma \ref{16:09}.
Taking $u_n$ as a test function in \eqref{15:56}, we get
\be\label{tenerife}
\alpha \io|\nabla u_n|^2 \le \io\frac{\vn}{(\frac1n+u_n)^{\gamma}}\le |\Omega|C_0^{1-\gamma},
\ee
where we have used that $s\to \frac{s}{(\frac1n+s)^{\gamma}}$ is increasing for $s\ge0$ and that $0\le u_n\le C_0$.
Hence there exists $u\in W^{1,2}_0(\Omega)\cap\elle{\infty}$ such that $u_n \to u$ weakly in $W^{1,2}_0(\Omega)$, strongly in any $\elle q$ with $1 \leq q < \infty$ (recall the uniform bound \eqref{14;25}) and a.e. in  $\Omega$. 
Passing to the limit in estimate \eqref{14;25} gives us the $L^{\infty}-$bound for $u$.\\
Take now $\phi\in C^1_c(\Omega)$ as a test function in \eqref{15:56}. Property \eqref{bfb} implies that 
\[
\io A(x) \nabla u\nabla \phi=\lim_{n\to\infty}\io A(x)\nabla u_n\nabla \phi= \lim_{n\to\infty}\io \frac{(1-\psi u_n^{\gamma+r-1})}{(\frac1n+u_n)^{\gamma}}\phi =\io \frac{1}{u^{\gamma}}(1-v u^{\gamma+r-1})\phi.
\]
To show that indeed $u$ satisfies \eqref{22:03}, let us use the argument introduced in Theorem 2.2 of \cite{bd}. Given $\phi \in W^{1,2}_0(\Omega)$, we consider $\{\phi_n\}\subset C^1_c(\Omega)$ such that $\phi_n\to \phi$ strongly in $W^{1,2}_0(\Omega)$. Taking $\varphi=(\eps^2+|\phi_n-\phi_m|^2)^{\frac12}-\eps$ as a test function, we get
\[
0\le \io \frac{1}{u^{\gamma}}(1-v u^{\gamma+r-1})[(\eps^2+|\phi_n-\phi_m|^2)^{\frac12}-\eps]=\io A(x) \nabla u\nabla (\phi_n-\phi_m)\frac{\phi_n-\phi_m}{(\eps^2+|\phi_n-\phi_m|^2)^{\frac12}}
\]
\[
\le \beta \| u\|_{W^{1,2}_0(\Omega)}\| \phi_n-\phi_m\|_{W^{1,2}_0(\Omega)}.
\]
Using Fatou Lemma w.r.t. $\eps\to0$, it follows
\[
0\le \io \frac{1}{u^{\gamma}}(1-v u^{\gamma+r-1})\vert\phi_n-\phi_m\vert \le \beta \| u\|_{W^{1,2}_0(\Omega)}\| \phi_n-\phi_m\|_{W^{1,2}_0(\Omega)},
\]
namely the sequence $\frac{1}{u^{\gamma}}(1-v u^{\gamma+r-1})\phi_n$ is Cauchy in $\elle1$. That in turn implies
\[
\frac{\phi}{u^{\gamma}}\in\elle 1,\ \ \ \io A(x)\nabla u\nabla\phi+\io v u^{r-1}\phi =\io\frac{\phi}{u^{\gamma}} \ \ \ \forall \ \phi\in W^{1,2}_0(\Omega)
\]
This conclude the proof of existence.\\
To deal with uniqueness, let us assume that there exists $\tilde u\in W^{1,2}_0(\Omega)$ that satisfies both \eqref{22:03} and \eqref{bound}.
Then we would get
\[
\alpha\io |\nabla (u-\tilde u)_+|^2\le \io \left[\frac{(1- v  \ u^{\gamma+r-1} )}{(\frac1n+ u)^{\gamma}}-\frac{(1-v \ \tilde u^{\gamma+r-1})}{(\frac1n+u_n)^{\gamma}}\right](u-\tilde u)_+
\] 
\[
=\int_{w_n\ge u_n} \left[\frac{(1- v \  u^{\gamma+r-1} )}{(\frac1n+ u)^{\gamma}}-\frac{(1-v \ \tilde u^{\gamma+r-1})}{(\frac1n+\tilde u)^{\gamma}}\right](u-\tilde u)_+\le 0,
\]
 where the last inequality comes from the fact that, a.e. in $ \Omega$, the function
\[
s\to \frac{1-\psi(x)s^{\gamma+r-1}}{s^{\gamma}}
\]
is strictly decreasing for $0<s\le \|v\|_{\elle{\infty}}^{\frac1{1-r-\gamma}}$. Then $\tilde u\equiv u$ and the proof is complete.\\
\end{proof}

\noindent{\bf Proof of Theorem \ref{existence1}-Completed.}

\begin{proof}

Let us define a map $F:W^{1,2}_0(\Omega)\to W^{1,2}_0(\Omega)$ as follows: for any $z\in W^{1,2}_0(\Omega)$ let $v_z\in W^{1,2}_0(\Omega)\cap L^{\infty}(\Omega)$ be the unique solution of 
\[
\io M(x)\nabla v_z \nabla \phi=\io T_{\sigma}(|z|)^{r}\phi
\ \ \ \forall \ \phi\in W^{1,2}_0(\Omega),
\]
for some $\sigma>0$ to be properly chosen, $T_{\sigma}$ as in \eqref{25a}, and $F(z)=u_z\in W^{1,2}_0(\Omega)\cap L^{\infty}(\Omega)$ the unique solution to 
\be \label{helpbis}
 \io A(x)\nabla u_z\nabla\phi+\io{v_z}{ u^{r-1}_z} \phi=\io\frac{\phi}{u^{\gamma}_z}\ \ \ \forall \ \phi\in W^{1,2}_0(\Omega)
\ee
such that  $ \|u_z\|_{\elle{\infty}}\le \min\{\|v_z\|_{\elle{\infty}}^{\frac1{1-r-\gamma}},C_0\}$ and $\|u_z\|_{W^{1,2}_0(\Omega)}\le |\Omega|C_0^{1-\gamma}$. The existence of $v_z$ is classical, while the existence of $u_z$ is assured by Lemma \ref{primo0}. Moreover, thanks to classical elliptic estimates (see for instance \cite{bc}), we know that there exists a positive constant $K(\sigma)$ such that 
\be\label{14:08}
\|v_z\|_{\elle{\infty}}+\|v_z\|_{W^{1,2}_0(\Omega)}\le K(\sigma).
\ee
Thanks to the energy bound of $u_z$ (see \eqref{bound}), we deduce that $F(B(\tilde C))\subset B(\tilde C)$, where 
$$B(\tilde C)=\{z\in W^{1,2}_0(\Omega) \ : \ \|z\|_{ W^{1,2}_0(\Omega)}\le \tilde C = |\Omega|C_0^{1-\gamma}\}.$$
We now claim that the map $z\to F(z)=u_z$ is completely continuous. In order to prove it, let $\{z_n\}\subset W^{1,2}_0(\Omega)$ be a sequence weakly convergent to $z\in W^{1,2}_0(\Omega)$. It easily follows that  $v_{z_n}\to v_{z}$ strongly in $W^{1,2}_0(\Omega)$. On the other hand, notice that each $u_{z_n}$ solves
\be \label{eqn}
 \io A(x)\nabla u_{z_n}\nabla\phi+\io{v_{z_n}}{ u_{z_n}^{r-1}} \phi=\io\frac{\phi}{u_{z_n}^{\gamma}}\ \ \ \forall \ \phi\in W^{1,2}_0(\Omega)
\ee
and that (by definition of $u_{z_n}$)
\be\label{help}
 \|u_{z_n}\|_{\elle{\infty}}\le \min\{\|v_{u_n}\|_{\elle{\infty}}^{\frac1{1-r-\gamma}},C_0\},\ \ \ \ \|u_{z_n}\|_{W^{1,2}_0(\Omega)}\le |\Omega|C_0^{1-\gamma}.
\ee
Then, we have that, up to a subsequence, $u_{z_n}\to \xi$ weakly in $W^{1,2}_0(\Omega)$ and $a.e.$ in $\Omega$ for some $\xi \in W^{1,2}_0(\Omega)$. Now we want to prove that $\xi\equiv u_z$.  For any $0\le  \phi\in W^{1,2}_0(\Omega)$, it follows
\[
\io A(x)\nabla \xi \nabla \phi=\lim_{n\to\infty}\io A(x)\nabla u_{z_n} \nabla \phi=\lim_{n\to\infty} \io \frac{1}{u_{z_n}^{\gamma}}(1-v_{z_n}u_{z_n}^{\gamma+r-1})\phi
\]
\[
\ge \io \frac{1}{\xi^{\gamma}}(1-v_{z}\xi^{\gamma+r-1})\phi \ \ \ \forall  \ 0\le  \phi\in W^{1,2}_0(\Omega),
\]
where the last inequality follows from Fatous' Lemma. This chain of inequalities provides us with two pieces of information. The first one is that
\be\label{23:00}
0\le \io \frac{1}{\xi^{\gamma}}(1-v_{z}\xi^{\gamma+r-1})\phi<\infty \ \ \ \forall  \ 0\le  \phi\in W^{1,2}_0(\Omega).
\ee
The second one is that
\be\label{23:00bis}
\io A(x)\nabla \xi \nabla \phi\ge \io \frac{1}{\xi^{\gamma}}(1-v_{z}\xi^{\gamma+r-1})\phi\ \ \ \forall  \ 0\le  \phi\in W^{1,2}_0(\Omega).
\ee
To prove the reverse inequality, notice that any $u_{z_n}$ satisfies
\[
\io A(x)\nabla u_{z_n} \nabla \phi=\io  \frac{1}{u_{z_n}^{\gamma}}(1-v_{z_n}u_{z_n}^{\gamma+r-1})\phi\ge  \io \frac{1}{(1+u_{z_n})^{\gamma}}(1-C(\sigma) u_{z_n}^{\gamma+r-1})\phi 
\]
for all $0\leq \phi\in W^{1,2}_0(\Omega),$ where we have used the $L^{\infty}$-bound of \eqref{14:08} to obtain the last inequality. Then any $u_{z_n}$ satisfies the assumption of Lemma \ref{16:51} and the bound from below \eqref{bfb} holds true uniformly with respect to $n$. This implies that
\be\label{bfbbis}
\forall  \ \omega\subset\subset \Omega \ \ \exists \ c_{\omega}>0 \ : u_{z_n}\ge c_{\omega} \ \ \ \mbox{and} \ \   \ \xi\ge c_{\omega}.
\ee
Then $\phi \frac{u_{z_n}}{\xi}$, with $0\le  \phi\in C^1_c(\Omega)$, is well defined and an admissible test function for \eqref{eqn}. It results
\[
 \io A(x)\nabla u_{z_n}\nabla \phi\frac{u_{z_n}}{\xi}+\io A(x)\nabla u_{z_n}\nabla u_{z_n} \frac{\phi}{\xi}
\]
\[
=\io \frac{1}{u_{z_n}^{\gamma}}(1-v_{z_n}u_{z_n}^{\gamma+r-1})\phi\frac{u_{z_n}}{\xi}+\io A(x) \nabla u_{z_n}\nabla \xi\frac{u_{z_n}}{\xi^2}\phi.
\]
Thanks to \eqref{help}, the weak convergence in $W^{1,2}_0(\Omega)$, the a.e. convergence of $u_{z_n}$ to $u_{z}$,  \eqref{bfbbis}, we can take the liminf on the left hand side above and the limit on the right hand one w.r.t. $n\to\infty$, to obtain
\be\label{23:00tris}
\io A(x)\nabla \xi\nabla \phi \le \io \frac{1}{\xi^{\gamma}}(1-v_{z}\xi^{\gamma+r-1})\phi \ \ \ \forall \ 0\le  \phi\in C^1_c(\Omega).
\ee
From inequalities \eqref{23:00bis} and \eqref{23:00tris} and \eqref{23:00} it is standard to infer
\[
\io A(x)\nabla \xi\nabla \phi = \io \frac{1}{\xi^{\gamma}}(1-v_{z}\xi^{\gamma+r-1})\phi \ \ \ \forall \   \phi\in W^{1,2}_0(\Omega).
\]
Moreover, passing to the limit in the estimates \eqref{help}, we deduce at first that  $ \|\xi\|_{\elle{\infty}}\le \min\{\|v_z\|_{\elle{\infty}}^{\frac1{1-r-\gamma}},C_0\}$ and then, thanks to the uniqueness assured by Lemma \ref{primo0}, that $\xi$ has to coincide with $u_z$.\\ This proves that the map $z\to F(z)=u_z$ is completely continuous. Then the Schauder's Fixed Point Theorem assure us that there exist nonegative $(u,v)\in (W^{1,2}_0(\Omega)) \cap L^{\infty}(\Omega)^2$ such that
\[
\begin{cases}
\medskip
\medskip

\displaystyle
 \io A(x)\nabla u\nabla\phi+\io{v}{ u^{r-1}} \phi=\io\frac{\phi}{u^{\gamma}}\ \ \ &\forall \ \phi\in W^{1,2}_0(\Omega)\\
\medskip
\medskip

\displaystyle \io M(x)\nabla v \nabla \phi=\io T_{\sigma}(u)^{r}\phi
\ \ \ &\forall \ \phi\in W^{1,2}_0(\Omega).
\end{cases}
\] 
Since $u$ satisfies the $L^{\infty}$ estimate in  \eqref{help}, it is enough to take $\sigma> C_0$ (where $C_0$ is the absolute constant given by Lemma \ref{16:09}) from the beginning  so that $T_{\sigma}(u)=u$.  It follows from Lemma \ref{16:51} that $u>0$ in $\Omega$, whence implies by the second equation  that $v>0$ in $\Omega$.\\

In order to deal with teh uniqueness, let us assume $r\ge1$ and that there exists another solution $(\tilde u, \tilde v)\in (W^{1,2}_0(\Omega)) \cap L^{\infty}(\Omega)^2$. We follow the approach of \cite{cgh}. Assumin Taking $u-\tilde u$ as a test function in the equations solved by $u$ and $\tilde u$, we get
\be\label{10:47}
\alpha \io|\nabla (u-\tilde u)|^2+\io \big(vu^{r-1}-\tilde v\tilde u^{r-1}\big) (u-\tilde u)
\ee
\[
\le  \io A(x)\nabla (u-\tilde u)\nabla (u-\tilde u)+\io \big(vu^{r-1}-\tilde v\tilde u^{r-1}\big) (u-\tilde u)
\]
\[
\le\io\left(\frac{1}{u^{\gamma}}-\frac{1}{\tilde u^{\gamma}}\right)(u-\tilde u)\le 0.
\]
On the other side, taking $(v-\tilde v)$ as a test function in the equations solved by $v$ and $\tilde v$, it results
\[
\alpha \io|\nabla (v-\tilde v)|^2\le \io M(x)\nabla (v-\tilde v)\nabla (v-\tilde v)=\io(u^r-\tilde u^r)(v-\tilde v).
\]
To estimate the left hand side above with the second term in \eqref{10:47} notice that
\[
 \big(vu^{r-1}-\tilde v\tilde u^{r-1}\big) (u-\tilde u)=vu^{r}+\tilde v\tilde u^{r}-vu^{r-1}\tilde u-\tilde v\tilde u^{r-1} u
\]
\[
\ge_1 \frac 1r\left(vu^{r}+\tilde v\tilde u^{r}-v\tilde u^r-\tilde v u^r  \right)=\frac 1r(u^{r}-\tilde u^{r})(v-\tilde v ),
\]
where inequality $\le_1$ comes from the Young inequality $a^{r-1}b\le \frac{1-r}{r}a^r+\frac1rb^r$ applied to $u^{r-1}\tilde u$ and $\tilde u^{r-1} u$.
The previously obtained inequalities implies
\[
\alpha \io|\nabla (u-\tilde u)|^2+\frac{\alpha}{r} \io|\nabla (v-\tilde v)|^2\le 0,
\]
namely the solution is unique.\\

Finally, let us proof $ii)$.
Consider the  functional $J:\left(W^{1,2}_0(\Omega)\cap\elle{\infty}\right)^2\to \mathbb R$ defined by
\be\label{func}
J(w,z)=
\begin{cases}
\displaystyle
 \frac12\io \big(A(x)\nabla w\nabla w-\frac1{r} M(x)\nabla z\nabla z\big)
+\io \left( \frac{z^+ |w|^r}{r}
-\frac{\bigl(w^{+} \bigr)^{1-\gamma}}{1-\gamma}\right)  \  &\mbox{if finite}
\\
+\infty  \ &\mbox{otherwise}
\end{cases}
\ee
with $\gamma < 1$.

Note that if $(u,v)\in \left(W^{1,2}_0(\Omega)\cap\elle{\infty}\right)^2$ is a energy solution to \eqref{sing}, then $J(u,v)<+\infty$ and the chain of inequalities \eqref{16} are well defined. Let us prove then at first that
\be\label{05.07}
J(u,z)\le J(u,v) \ \ \ \mbox{ for any } \ \ \ z\in W^{1,2}_0(\Omega),
\ee
namely $v$ is a maximum for the functional $z\to J(u,z)$. Let us take $\frac 1r (v-z^+)$ as a test function in the second equation of \eqref{energysoleq} to get
$$
\frac 1r\io M(x)\nabla v \D(v-z^+)\pm\frac1{2r}\io M(x)\nabla z^+\nabla z^+ = \frac 1r\io | u|^r (v-z^+)\,.
$$
After simple manipulations, we get
\[
\frac{1}{2r}\io M(x)\nabla z^+\nabla z^+-\frac 1r\io z^+ | u|^r 
\]
\[
=\frac{1}{2r}\io M(x)\nabla v\nabla v-\frac 1r\io v | u|^r +\frac{1}{2r}\io M(x)\nabla (v-z^+)\nabla (v-z^+)
\]
\[
\ge \frac{1}{2r}\io M(x)\nabla v\nabla v-\frac 1r\io v | u|^r
\]
At this point, changing sing in the inequality above and adding the finite term
\[
\frac12 \io A(x)\nabla u\nabla u-\frac1{1-\gamma}\io u^{1-\gamma}
\]
we obtain \eqref{05.07} (recall that $u,v>0$  in $\Omega$). To conclude the proof of the Theorem we have to show that
\be\label{05.07bio}
J(u,v)\le J(w,v) \ \ \ \mbox{ for any } \ \ \ w\in W^{1,2}_0(\Omega),
\ee
namely the functional $w\to J(w,v)$ has a minimum at $u$. Let us assume that the right hand side above is finite, otherwise there is nothing to prove. Moreover we can also assume $w\ge0$, since 
\[
\frac12\io A(x)\nabla w^+\nabla w^+
+\io \left( \frac{1}{r}z^+ (w^+)^r
-\frac{1}{1-\gamma}\bigl(w^{+} \bigr)^{1-\gamma}\right)
\] 
\[
\le  \frac12\io A(x)\nabla w\nabla w
+\io \left( \frac{1}{r}z^+ |w|^r
-\frac{1}{1-\gamma}\bigl(w^{+} \bigr)^{1-\gamma}\right).
\]
Then \eqref{05.07bio} is equivalent to prove
\[
\io A(x)\nabla u\nabla u+\io Q_v(u)\le \io A(x)\nabla w\nabla w+\io Q_v(w), \ \ \ \mbox{for} \ \ 0\le w\in W^{1,2}_0(\Omega),
\]
where the real valued function $t\to Q_v(t)$ is defined as
\[
Q(t)=\frac1r\,v\,t^{r} - \frac{1}{1-\gamma}t^{1-\gamma}, \ \ \ \mbox{for} \ \ t>0.
\]
The convexity of $t\to Q_v(t)$ (due to  $\gamma >0$ and $r \geq 1$), namely
$$
Q(t)\geq Q(s) + Q'(s)(t-s),
$$
implies (with $t=w(x)$, $s=u(x)$) that 
\be\label{unouno}
\io\frac1r\,v \, w^{r}  
- \io\frac{1}{1-\gamma}w^{1-\gamma}
\ee
\[
\geq 
\io\frac1r\,v\,u^r 
-\io \frac{1}{1-\gamma}u^{1-\gamma}
 + \io\Big[ 
 v\,u^{r-1} -u^{-\gamma}
 \Big](w-u).
\]
Taking now $(u-w)$ as a test function in the first equation of \eqref{energysoleq} we deduce that
\be\label{duedue}
0=\io A(x)\nabla u\nabla(u-w)+\io \left(vu^{r-1}-\frac{1}{u^{\gamma}}\right)(u-w)
\ee
\[
=\frac12\io A(x)\nabla u\nabla u- \frac12\io A(x)\nabla w\nabla w+\frac12\io A(x)\nabla (u-w)\nabla (u-w)
\]
\[
+\io \left(vu^{r-1}-\frac{1}{u^{\gamma}}\right)(u-w).
\]
Putting together \eqref{unouno} and \eqref{duedue} and neglecting the positive term $\frac12\io A(x)\nabla (u-w)\nabla (u-w)$, we get the desired inequality and finish the proof of $ii)$ and of Theorem.

\end{proof}

Now we will allow bigger values for the parameter $\gamma$ and double singularities for a single equation paying the price of some regularity of $\partial \Omega$.

\begin{lemma}\label{08.7}
Let us assume \eqref{alpha}, \eqref{30.08}, $\gamma\in(1,3)$, $r>1-\gamma$ and $0\le v\in\elle{\infty}$.
Then there exists a unique $u\in W^{1,2}_0(\Omega)\cap\elle{\infty}$ satisfying
\begin{equation}\label{boundbis}
\left|\frac{\phi}{u^{\gamma}}\right|\in\elle 1 \ \ \ \  \|u\|_{\elle{\infty}}\le \min\{\|v\|_{\elle{\infty}}^{\frac1{1-r-\gamma}},C_0\},\ \ \ \ \|u\|_{W^{1,2}_0(\Omega)}\le  C_1,
\end{equation}
\be\label{09;05}
\ \ \mbox{and} \ \ \ \ \ \ \io A(x)\nabla u\nabla\phi+\io v u^{r-1}\phi =\io\frac{\phi}{u^{\gamma}} \ \ \ \forall \ \phi\in W^{1,2}_0(\Omega),
\ee
 where $C_0, C_1$ are positive real constants. Moreover
\be\label{18;25}
u\ge c\, d^{\tau}
\ee
where $\tau={2}/{(\gamma+1)}$ and $c$ is the one of \eqref{tenerifebis}.
\end{lemma}
\begin{proof}
Let us prove at first that the sequence of approximating solutions $\{u_n\}$ given by Lemma \ref{16:09} is bounded in $W^{1,2}_0(\Omega)$ even for the considered range of values for $\gamma$. Notice that, since $\gamma>1$, we cannot simply obtain estimate \eqref{tenerife} as in Lemma \ref{primo0}. Instead we are going to take advantage of the quantitative estimate from below provided by Lemma \eqref{regular} (here we need the regularity assumption on $\partial\Omega$). 
Taking $ u_n$ as a test function in \eqref{15:56}, it results that
\begin{equation}\label{hhardy}
\alpha\io|\nabla  u_n|^2 \le \io\frac{ u_n}{(\frac1n+ u_n)^\gamma}\le_1 \tilde C \io\frac{ u_n}{d^{\tau\gamma}}
\end{equation}
\[
\le\tilde C_2\left(\io\frac{ u_n^2}{d(x)^2}\right)^{\frac12}\left(\io d(x)^{2-2\tau\gamma}\right)^{\frac12},
\]
where we have used inequality \eqref{tenerifebis} in $\le_1$ and H\"older inequality in $\le_2$.  Let us recall that, thanks to the smoothness of $\partial \Omega$ and the fact that  $2-2\tau\gamma>-1$ (since $\gamma<3$), we have that
\[
\io d(x)^{2-2\tau\gamma}\le D \ \ \ \mbox{for some constant} \ \ \ D>0.
\]
Moreover, using again our regularity assumption on $\partial\Omega$, Hardy inequality assures
\[
\io\frac{ u_n^2}{d(x)^2}\le \mathcal{H} \io|\nabla u_n|^2
\]
for some absolute constant $\mathcal{H}$  (see for instance \cite{bm}) .
Plugging the last two pieces of information in \eqref{hhardy} we get
\[
\alpha\io|\nabla  u_n|^2 \le D \mathcal{H}.
\]
Then we deduce that there exists $u\in W^{1,2}_0(\Omega)\cap\elle{\infty}$ such that, up to a subsequence, $u_n \to u$ weakly in $W^{1,2}_0(\Omega)$, strongly in any $\elle q$ with $1\leq q <2^*$ 
and a.e. in  $\Omega$. To prove that such function is the unique that satisfies \eqref{boundbis} and \eqref{09;05}, we follow by the same argument of Lemma \ref{primo0} and we omit the details. Finally \eqref{18;25} is obtained simply taking the limit as $n\to\infty$ in \eqref{tenerifebis}.\\
\end{proof}

\noindent{\bf Proof of Theorem \ref{30.08x}-completed.}
\begin{proof}
The proof follows exactly the same argument as in the proof of Theorem \ref{existence1}. The only difference is that we use the estimates provided by Lemma \ref{08.7} instead of Lemma \ref{primo0}

\end{proof}

\section{Outside the energy space}

In this section, we consider stronger singularities that prevents the solution to belong to the energy space and so we consider a function $u\in  W_{loc}^{1,2}(\Omega)$ being zero on the boundary of $\Omega$ whenever $(u-\epsilon)^+ \in W_{0}^{1,2}(\Omega)$ for any $\epsilon>0$. 

\begin{lemma}\label{14;45}
Let $\Omega$ be a bounded open set of $\mathbb{R}^N$, let us assume \eqref{alpha}, $\gamma>0$ and $r>1-\gamma$.
For any $n\in \mathbb N$, there exist a couple of nonegative functions $(u_n,\psi_n)\in \left(W^{1,2}_0(\Omega)\cap\elle{\infty}\right)^2$ satisfying 
\begin{equation}
\label{4312}
\begin{cases} 
\medskip
\medskip

\displaystyle \io A(x)\nabla u_n\nabla\phi+\io v_n u_n^{r-1} \phi=\io\frac{\phi}{(u_n+1/n)^{\gamma}}, & \forall \ \phi\in W^{1,2}_0(\Omega) \\

\medskip

\displaystyle  \io M(x)\nabla v_n \nabla \phi=\io u_n^r\phi & \forall \ \psi\in W^{1,2}_0(\Omega).
\end{cases}
\end{equation}
Moreover there exists $\tilde C=\tilde C(\alpha, N, \Omega, \gamma, r)$ such that
\be \label{uuno}
\|u_n\|_{\elle{\infty}}+\|v_n\|_{\elle{\infty}}\le \tilde C
\ee
\end{lemma}
\begin{proof}
The proof of the Lemma is simplified version of Theorem \ref{existence1} , since for any $n\in\mathbb N$ fixed the left hand side of the first equation of \eqref{4312} is not singular. We just give a quick sketch for the convenience of the reader. The  strategy is to apply the Schauder fixed point Theorem. For it, consider the map $F:W^{1,2}_0(\Omega)\to W^{1,2}_0(\Omega)$ defined as follows: for any $z\in W^{1,2}_0(\Omega)$ let $v_z\in W^{1,2}_0(\Omega)$ be the unique solution of 
\[
\io M(x)\nabla v_z \nabla \phi=\io T_{\sigma}(|z|)^{r}\phi
\ \ \ \forall \ \phi\in W^{1,2}_0(\Omega),
\]
for some $\sigma>0$ to be chosen, and $F(z)=u_z\in W^{1,2}_0(\Omega)$ the unique solution to 
\be \label{helpbis}
 \io A(x)\nabla u_z\nabla\phi+\io{v_z}{ u^{r-1}_z} \phi=\io\frac{\phi}{(u_z+1/n)^{\gamma}}\ \ \ \forall \ \phi\in W^{1,2}_0(\Omega)
\ee
given by Lemma \ref{15:56}. We recall that $v_z$ and $u_z$ satisfy \eqref{14:08} and \eqref{14;25} respectively.
As in the proof of Theorem \ref{existence1}, we deduce that the map $z\to F(z)=u_z$ is invariant for some ball $B\subset W^{1,2}_0(\Omega)$, whose radius does not depend on $\sigma$, and it is completely continuous. Applying Schauder fixed point Theorem and taking $\sigma$ large enough we obtain the existence of the searched solutions. To conclude, notice that Lemma \ref{16:09} implies that $|u_n|\le C_0$ (see estimate \eqref{14;25}) and then estimate \eqref{uuno} follows easily.
\end{proof}

\noindent{\bf Proof of Theorem \ref{teo15}-completed.}

\begin{proof}
Let $\{(u_n,v_n)\}\subset  \left(W^{1,2}_0(\Omega)\cap\elle{\infty}\right)^2$ be the sequence of functions provided by Lemma \ref{14;45}.
Using $v_n$ as a test function in the second equation of \eqref{4312}, we get 
\[
\displaystyle \alpha \io |\nabla v_n|^2 \leq \io u_n^rv_n \leq  C_0^r\io v_n,
\]
that easily implies that $\{v_n\}$ is bounded in $W_{0}^{1,2}(\Omega)$. Then there exists $ v \in W_{0}^{1,2}(\Omega)$ such that
$$v_n \rightharpoonup v ~\mbox{in }W_{0}^{1,2}(\Omega),~~v_n \to  v ~\mbox{in }L^2(\Omega),~\mbox{and }v_n(x) \to v (x)~\mbox{a.e. in }\Omega.$$
To deal with the second equation in \eqref{4312}, let us notice that thanks to \eqref{uuno}, we can apply Lemma \ref{16:51} with $M=\tilde C$ to deduce that 
\be\label{duue}
\forall \ \omega\subset\subset \Omega \ \ \exists \ c_{\omega, M} \ : \ u_n\ge  c_{\omega, M} \ \ \mbox{in} \ \ \omega   \ \ \ \forall n>0.
\ee
Take now any open set $U \subset \subset \Omega$ and $\xi \in C_c^{\infty}(\Omega)$ such that  $  0 \leq \xi \leq 1$ and $\xi = 1$ in $U$. By using $u_{n}\xi^2$ as a test function in the first equation in \eqref{4312}, we obtain
\begin{eqnarray}\label{11}
\displaystyle \int_{\Omega} |\nabla u_{n}|^2 \xi^2dx + 2\displaystyle\int_{\Omega} \nabla u_{n} \nabla \xi (u_{n} \xi )dx \leq  \displaystyle\int_{\Omega} \frac{u_n}{( u_{n} +1/n)^{\gamma}} \xi^2 dx.
\end{eqnarray} 
By using \eqref{uuno} and  Young's inequality, we obtain that
\begin{eqnarray}\label{21}
\displaystyle\int_{\Omega} \nabla u_{n} \nabla \xi (u_{n}\xi )dx & \leq & \displaystyle\int_{\Omega} |\nabla u_{n}||\nabla \xi| u_{n}\xi dx  \nonumber \\
& \leq & \theta \displaystyle\int_{\Omega} \left(|\nabla u_{n}|\xi\right)^{2}dx + C(\theta) \displaystyle\int_{\Omega} u_{n}^2 |\nabla \xi|^2 dx \nonumber \\
& \leq & \theta\displaystyle\int_{\Omega} \xi^2 |\nabla u_{n}|^2 dx + C(\theta),  
\end{eqnarray}
where $C(\theta)$ is a cumulative positive constant with $\theta>0$. Then \eqref{11} becomes 
$$  \displaystyle\int_{U} |\nabla u_{n}|^2 dx \leq \displaystyle\int_{\Omega} |\nabla u_{n}|^2 \xi^2dx \leq  C(\theta).$$
This implies that $\{u_{n}\}$ is bounded in $W_{{\mathrm{loc}}}^{1,2}(\Omega)$. So, there exists $u \in W_{\mathrm{loc}}^{1,2}(\Omega)\cap L^{\infty}(\Omega)$ such that
\begin{equation}
\label{121}
 \left\{
\begin{array}{l}
 u_{n} \rightharpoonup  u \ \ \mbox{in}  \  \ W^{1,2}(U) \\
 u_{n} \rightarrow u \ \ \ \mbox{in}  \ \ L^{2}(U)\\
u_{n}(x) \rightarrow u(x) \ \ \mbox{almost everywhere in} \ \Omega ,

\end{array}
\right.
\end{equation}
for each  $U \subset \subset \Omega$ given. In particular, 
\begin{equation}
\label{49}
u_n^r \to u^r~\mbox{in} ~ L^2(\Omega),
\end{equation}
Recalling \eqref{uuno}, we can take advantage of the Lebesgue Theorem to pass to the limit in the equation solved by $v_n$ to obtain that
\[
\displaystyle  \io\nabla  v  \nabla \phi=\io u^r\phi,\ \ \forall \ \phi\in W^{1,2}_0(\Omega).
\]
Finally, thanks to \eqref{duue} and \eqref{121}, we can pass to the limit in the first equation of \eqref{4312} to conclude that 
$$\displaystyle \io\nabla u\nabla\phi+\io\varphi u^{r-1} \phi=\io\frac{\phi}{u^{\gamma}} ~~\mbox{for any }\phi \in C^{\infty}_{c}(\Omega).$$

To finish the proof, we just need to prove that $(u-\epsilon)^+ \in W^{1,2}_0(\Omega)$ for any $\epsilon>0$ given. Since $u_n$ satisfies \eqref{uuno} and \eqref{duue}, we can redo the above argument to infer that $(u_n-\epsilon)^+$ is bounded in $W_0^{1,2}(\Omega)$ for each $\epsilon>0$ so that $(u_n-\epsilon)^+ $ weakly converges to some $w \in W_0^{1,2}(\Omega)$. Since we already know from \eqref{121} that $u_n(x) \to u(x)$ almost everywhere in $\Omega$, we must have $(u-\epsilon)^+=w \in W_0^{1,2}(\Omega)$. This ends the proof.\\
\end{proof}

\noindent{\bf Acknowledgments.}  
The first author was supported by PNPD/CAPES-UnB-Brazil, Grant 88887.363582/2019-00, and by the Austrian Science Fund (FWF) projects F65 and P32788. The third author was supported by CNPq/Brazil with grant 311562/2020 - 5.

\end{document}